\documentclass[xcolor=x11names,reqno,12pt]{amsart}
 \usepackage[labelsep=space]{caption}
 \usepackage{amsmath}
 \usepackage{amsthm}
 \usepackage{amssymb}
 \usepackage{latexsym,longtable}
 \usepackage{multicol}
 \usepackage{mathrsfs}
 \usepackage[vcentermath,enableskew,stdtext]{youngtab}
  \usepackage[left=1.1in,right=1.1in,top=1.08in,bottom=1.1in]{geometry}
 \usepackage[all]{xy}
    \SelectTips{cm}{10}     
    \everyxy={<2.5em,0em>:} 
 \usepackage{fancyhdr}      
 \usepackage{enumitem}
 \usepackage{bm}
 \usepackage{stmaryrd}
 \usepackage{etex}
 \usepackage{float}
 \usepackage{array}
 \usepackage{colortbl}
 \usepackage{mathtools}
 \restylefloat{figure}
\numberwithin{equation}{section}
\usepackage{url}

\usepackage{amsfonts,epsfig,color}
\makeatletter
\newcommand*{\centerfloat}{%
  \parindent \z@
  \leftskip \z@ \@plus 1fil \@minus \textwidth
  \rightskip\leftskip
  \parfillskip \z@skip}
\makeatother
\newcounter{ctr}
\theoremstyle{plain}
\newtheorem{theorem}{Theorem}[section]
\newtheorem{lemma}[theorem]{Lemma}
\newtheorem{corollary}[theorem]{Corollary}
\newtheorem{proposition}[theorem]{Proposition}

\theoremstyle{definition}

\newtheorem{remark}[theorem]{Remark}

\newcommand{\ignore}[1]{}

\newcommand{\B}{\ensuremath{\mathcal{B}}}

\newcommand{\CC}{\ensuremath{\mathbb{C}}}

\newcommand{\F}{\ensuremath{\mathcal{F}}}

\newcommand{\g}{\ensuremath{\mathfrak{g}}}
\newcommand{\gl}{\ensuremath{\mathfrak{gl}}}

\newcommand{\h}{\ensuremath{\mathfrak{h}}}
\renewcommand{\hom}{\text{\rm Hom}}

\newcommand{\QQ}{\ensuremath{\mathbb{Q}}}

\renewcommand{\sl}{\ensuremath{\mathfrak{sl}}}

\newcommand{\ZZ}{\ensuremath{\mathbb{Z}}}

\newcommand{\be}{\begin{equation}}
\newcommand{\ee}{\end{equation}}

\newcommand{\tsr}{\ensuremath{\otimes}}

\newcommand{\hatsl}{\ensuremath{\widehat{\sl}}}

\DeclareMathOperator{\wt}{wt}

\DeclareMathOperator{\aff}{aff}
\DeclareMathOperator{\cl}{cl}

\newcommand{\ce}{\ensuremath{\tilde{e}}} 
\newcommand{\cf}{\ensuremath{\tilde{f}}}
\newcommand{\myr}{r}
\newcommand{\bdr}{\ensuremath{\mathbf{r}}}

\newcommand{\crc}[1]{#1\star}

\newcommand{\sfb}{\ensuremath{\mathsf{b}}}

\newcommand{\nocdots}{}

\newcommand{\spa}{\hspace{.3mm}}

\newlength{\mycellsize}
\newcommand\mytbl[1]{
\vcenter{
\let\\=\cr
\baselineskip=-16000pt \lineskiplimit=16000pt \lineskip=0pt
\halign{&\mytblcell{##}\cr#1\crcr}}}

\newcommand{\mytblcell}[1]{{%
\def \arg{#1}\def \void{}%
\ifx \void \arg
\vbox to \mycellsize{\vfil \hrule width \mycellsize height 0pt}%
\else \unitlength=\mycellsize
\begin{picture}(1,1)
\put(0,0){\makebox(1,1){$#1\vphantom{\crc{#1}}$}}
\put(0,0){\line(1,0){1}}
\put(0,1){\line(1,0){1}}
\put(0,0){\line(0,1){1}}
\put(1,0){\line(0,1){1}}
\end{picture}%
\fi}}

\newlength{\cellsize}
\newcommand\mytableau[1]{
\vcenter{
\let\\=\cr
\baselineskip=-16000pt \lineskiplimit=16000pt \lineskip=0pt
\halign{&\mytableaucell{##}\cr#1\crcr}}}

\newcommand{\mytableaucell}[1]{{%
\def \arg{#1}\def \void{}%
\ifx \void \arg
\vbox to \cellsize{\vfil \hrule width \cellsize height 0pt}%
\else \unitlength=\cellsize
\begin{picture}(1,1)
\put(0,0){\makebox(1,1){$#1\vphantom{\crc{#1}}$}}
\put(0,0){\line(1,0){1}}
\put(0,1){\line(1,0){1}}
\put(0,0){\line(0,1){1}}
\put(1,0){\line(0,1){1}}
\end{picture}%
\fi}}

\newcommand\boldtableau[1]{
\vcenter{
\let\\=\cr
\baselineskip=-16000pt \lineskiplimit=16000pt \lineskip=0pt
\halign{&\boldtableaucell{##}\cr#1\crcr}}}

\newcommand{\boldtableaucell}[1]{{%
\def \arg{#1}\def \void{}%
\ifx \void \arg
\vbox to \cellsize{\vfil \hrule width \cellsize height 0pt}%
\else \unitlength=\cellsize
\begin{picture}(1,1)
\put(0,0){\makebox(1,1){$\mathbf{#1\vphantom{\crc{#1}}}$}}
\put(0,0){\line(1,0){1}}
\put(0,1){\line(1,0){1}}
\put(0,0){\line(0,1){1}}
\put(1,0){\line(0,1){1}}
\end{picture}%
\fi}}

\title{The DARK side of generalized Demazure crystals}

\begin{document}

\author{Jonah Blasiak}
\address{Department of Mathematics, Drexel University, Philadelphia, PA 19104}
\email{jblasiak@gmail.com}

\thanks{This work was supported by NSF Grant DMS-1855784.}

\begin{abstract}
Naoi \cite{NaoiVariousLevels} showed that tensor products of perfect Kirillov-Reshetikhin crystals
are isomorphic to certain generalized Demazure crystals.
We extend Naoi's results to address distinguished subsets of these tensor products.
In type A, these are naturally
described in terms of katabolizable tableaux which was key to
resolving conjectures
of Shimozono-Weyman \cite{SW} and Chen-Haiman~\cite{ChenThesis} in~\cite{BMP}.
\end{abstract}
\maketitle



\section{Introduction}
Naoi \cite{NaoiVariousLevels} showed that tensor products of perfect Kirillov-Reshetikhin (KR) crystals are isomorphic to certain generalized Demazure crystals introduced by
Lakshmibai-Littelmann-Magyar \cite{LakLitMag}.
From this  he obtained a Demazure operator formula for their characters using the well-developed theory of Demazure crystals
\cite{Josephdemazurecrystal, KashiwaraDemazure, LakLitMag, LittelmannCrystal}.
This formed a key step in his resolution of the $X=M$ conjecture \cite{NaoiXMConjecture} in type  $D_n^{(1)}$.

We extend Naoi's result to match a larger class of generalized Demazure crystals
with certain subsets of tensor products of perfect KR crystals, termed
 Kirillov-Reshetikhin affine Demazure (DARK) crystals.
Our result follows directly from techniques of \cite{NaoiVariousLevels},
but the deep combinatorial consequences shown for type $A$ in~\cite{BMP}
motivate this presentation of the results in the full generality of any nonexceptional type.

Naoi's work encompasses several earlier results connecting Demazure and KR crystals~\cite{FSS,SchillingTingley,Shimozonoaffine}.
The emphasis in these works is
on providing a model for KR crystals using the well-developed theory of Demazure crystals,
whereas here we are interested in using KR crystals to understand generalized Demazure crystals.
While there are combinatorial models of highest weight crystals for affine type~\cite{Kangyoungwalls,KKMMNN, KKMMNNDuke, LenartPostnikov, Littelmannpaths}
which lead to explicit descriptions of generalized Demazure crystals,
our explorations suggest that the combinatorics afforded by  DARK crystals is simpler.
This is possible because the isomorphism between generalized Demazure and DARK crystals is combinatorially nontrivial,
roughly analogous to the different models for the $U_q(\gl_n)$-highest weight crystal $B(\nu)$
afforded by semistandard tableaux of shape $\nu$ versus those provided by an
embedding $B(\nu) \hookrightarrow B(\lambda) \tsr B(\mu)$.

\section{Background on crystals}
We follow \cite{NaoiVariousLevels} almost completely and review the notation we will need,
emphasizing conventions which may not be well known.

\subsection{Affine Kac-Moody Lie algebras}
\label{ss affine kac moody lie algebras}

Let $\g$ be a complex affine Kac-Moody Lie algebra of
nonexceptional type (i.e., of type $A_n^{(1)}, B_n^{(1)}, C_n^{(1)}, D_n^{(1)}, A_{2n-1}^{(2)}, A_{2n}^{(2)},$ or $D_{n+1}^{(2)}$).
Let $I = \{0, 1, \ldots , n\}$ be the Dynkin nodes and $A = (a_{ij})_{i,j\in I}$ the Cartan matrix.
Let  $\h \subset \g$ be the Cartan subalgebra, which has a basis consisting of
the simple coroots $\{\alpha_i^\vee \mid  i \in  I\} \subset \h$ together with the scaling element $d \in \h$.
We have the linearly independent simple roots $\{\alpha_i \mid  i \in I\} \subset \h^*$, with pairings
$\langle \alpha_i^\vee, \alpha_j\rangle = a_{ij}$
and $\langle d, \alpha_i \rangle = \delta_{i 0}$ $(i,j \in I)$.
Let $(a_0, \dots, a_n)$ (resp. $(a_0^\vee, \dots, a_n^\vee))$ be the unique tuple of relatively prime positive
integers that give a linear dependence relation among the columns (resp. rows) of $A$.

Choose $N \in \ZZ_{\ge 1}$ and fundamental weights $\{\Lambda_i \mid i \in I\} \subset \h^*$
such that $\langle \alpha_i^\vee, \Lambda_j\rangle = \delta_{ij}$
 and $\langle d, \Lambda_j \rangle \in N^{-1} \ZZ$ for $i,j \in I$, the choices to be discussed further below.
Let $\delta = \sum_{i\in I} a_i\alpha_i$
be the null root. Note that
 $\{\Lambda_i \mid i \in I\} \sqcup \{\delta\}$ is a basis of $\h^*$
 and $\langle \alpha_i^\vee, \delta\rangle = 0$ for  $i \in I$ and $\langle d, \delta\rangle = a_0$.
Let
$P = \big \{\mu \in \h^* \mid \langle\alpha_i^\vee, \mu\rangle \in \ZZ \text{ for } i \in I, \langle d, \mu \rangle \in N^{-1}\ZZ\big\}
= \bigoplus_{i\in I} \ZZ \Lambda_i \oplus \ZZ \frac{\delta}{a_0 N} \subset \h^*$
be the weight lattice and
$P^+ = \sum_{i\in I} \ZZ_{\ge 0}\Lambda_i + \ZZ\frac{\delta}{a_0 N}$ the dominant weights.

Let  $\cl \colon  \h^* \to \h^*/\CC \delta$ be the canonical projection, and
set  $P_{\cl} = \cl(P) = \bigoplus_{i\in I} \ZZ \spa \cl(\Lambda_i)$.
Let $\aff \colon  \h^*/\CC\delta \to \h^*$ be the section of $\cl$ satisfying $\langle d, \aff(\mu) \rangle = 0$ for all $\mu \in h^* / \CC \delta$.
Set  $\varpi_i = \aff(\cl(\Lambda_i- a_i^\vee \Lambda_0))$ for  $i \in  I_0 := I \setminus \{0\}$.

The \emph{affine Weyl group} $W$ can be realized as the subgroup of $GL(\h^*)$ generated by the simple reflections $s_i$ $(i \in I)$, where $s_i$ acts by $s_i(\mu) = \mu - \langle \alpha_i^\vee, \mu \rangle \alpha_i$.  Let  $W_0$ be the subgroup generated by  $s_i$ for  $i \in I_0$.

Let $c_i = \max\{1, a_i/a_i^\vee\}$ for $i \in I_0$, and define $\widetilde{M} = \bigoplus_{i \in I_0} \ZZ c_i \varpi_i \subset P$.
For $\mu \in \widetilde{M}$, define the translation $t_\mu \in GL(\h^*)$ as in \cite[Equation 6.5.2]{KacBook} and set $T = \{t_\mu \mid \mu \in \widetilde{M}\}$.
These satisfy $t_\mu t_\lambda = t_{\mu+\lambda}$  and $w t_\mu w^{-1} = t_{w(\mu)}$ for  $w \in W_0$ and $\mu, \lambda \in \widetilde{M}$.
Thus  $\widetilde{W} = W_0 \ltimes T$  is a subgroup of $GL(\h^*)$, called the \emph{extended affine Weyl group}.

Let $\Sigma \subset \widetilde{W}$ denote the subgroup which takes the set $\{\alpha_i \mid i \in I\}$ to itself.
Thus each element $\tau \in \Sigma$ yields a permutation of $I$, which we also denote by $\tau$;
it is an automorphism of the Dynkin diagram, meaning that $a_{i j} = a_{\tau(i)\tau(j)}$ for all $i,j \in I$.
(See \cite[\S2.2, \S5.2]{NaoiVariousLevels} for the explicit description of $\Sigma$ as a set of permutations
in each type.)
There holds $\widetilde{W} \cong W \rtimes \Sigma$.
As discussed in \cite[\S2.2]{NaoiVariousLevels}, we can choose  $N$ and
$\langle d, \Lambda_i \rangle$
so that for all $\tau \in \Sigma$, $\tau(\Lambda_j) = \Lambda_{\tau(j)}$ for all $j \in I$ and $\tau(\delta) = \delta$.  Note that this implies $\widetilde{W}$ preserves  $P$.

\subsection{Crystals}
\label{ss crystals}

Let $U_q(\g)$ be the quantized enveloping algebra (as in \cite{Kas1}) specified by the above data
and the symmetric bilinear form  $(\cdot, \cdot) \colon  P \times P \to \QQ$ defined by $(\alpha_i, \alpha_j) = a_i^\vee a_i^{-1} a_{ij}$,  $(\alpha_i, \Lambda_0) = a_0^{-1}\delta_{i0}$,
$(\Lambda_0, \Lambda_0) = 0$.
It is generated by $e_i, f_i$,  $i \in I$,
and $q^h$,  $h \in P^* := \hom_\ZZ(P,\ZZ)$.  Let  $U'_q(\g)\subset U_q(\g)$ be
the subalgebra generated by $e_i, f_i$,  $i \in I$, and  $q^h$,  $h \in P_{\cl}^*=  \bigoplus_{i \in I}\ZZ \alpha_i^\vee$.
A \emph{$U_q(\g)$-crystal} (resp. $U'_q(\g)$-crystal) is
a set $B$ equipped with a \emph{weight function}
$\wt\colon B \to P$ (resp. $\wt\colon B \to P_{\cl}$) and \emph{crystal operators}
$\ce_i, \cf_i \colon B \sqcup \{0\} \to B \sqcup \{0\}$ ($i \in I$) such that for all $i \in I$ and $b \in B$, there holds
$\ce_i(0) = \cf_i(0) = 0$ and
\begingroup
\allowdisplaybreaks
\begin{align*}
&\wt(\ce_ib)= \wt(b) + \alpha_i \text{ whenever $\ce_ib \neq 0$}, \ \text{ and } \,
 \wt(\cf_ib)= \wt(b) - \alpha_i \text{ whenever $\cf_ib \neq 0$}; \\
\notag
& \varepsilon_i(b) := \max\{k \ge 0 \mid \ce_i^kb \neq 0\} < \infty, \quad \phi_i(b) := \max\{k \ge 0 \mid \cf_i^kb \neq 0\} < \infty;  \\
\notag
& \langle \alpha_i^\vee, \wt(b)  \rangle = \phi_i(b) - \varepsilon_i(b); \\
\notag
& \cf_i(\ce_ib) = b\text{ whenever $\ce_ib \neq 0$, \ \ \ and \ } \ce_i(\cf_ib) = b\text{ whenever $\cf_ib \neq 0$}.
\end{align*}%
\endgroup
This agrees with the notion of a seminormal crystal in \cite[\S7]{KashiwaraSurvey}.
We use the term \emph{crystal} to mean either a $U_q(\g)$-crystal or $U'_q(\g)$-crystal.

A \emph{strict embedding} of a crystal $B$ into a crystal $B'$ is an injective map
$\Psi \colon B \sqcup \{0\} \to B' \sqcup \{0\}$ such that $\Psi(0) = 0$  and
 $\Psi$ commutes with  $\wt$,   $\varepsilon_i$,  $\phi_i$,
$\ce_i$,  and $\cf_i$ for all  $i\in I$.  It is necessarily an isomorphism from  $B$ onto
a disjoint union of connected components of  $B'$.

For a $U_q(\g)$-crystal $B$ with weight function  $\wt \colon B \to P$, its \emph{$U_q'(\g)$-restriction} is the  $U_q'(\g)$-crystal with the same edges as  $B$ and
 weight function $\cl \circ \wt \colon  B \to P_{\cl}$.

For two crystals $B_1$ and $B_2$, their tensor product $B_1 \tsr B_2 = \{b_1 \tsr b_2 \mid b_1 \in B_1, b_2 \in B_2\}$ is
the crystal with weight function $\wt(b_1 \otimes b_2) = \wt(b_1) + \wt(b_2)$ and crystal operators
\begin{align}
\label{e crystal tensor 1}
\ce_i(b_1 \otimes b_2) &= \begin{cases}
\ce_ib_1 \otimes b_2 & \text{ if $\phi_i(b_1) \ge \varepsilon_i(b_2)$}, \\
b_1 \otimes \ce_ib_2 & \text{ if $\phi_i(b_1) < \varepsilon_i(b_2)$}.
\end{cases}\\
\label{e crystal tensor 2}
\cf_i(b_1 \otimes b_2) &= \begin{cases}
\cf_ib_1 \otimes b_2 & \text{ if $\phi_i(b_1) > \varepsilon_i(b_2)$}, \\
b_1 \otimes \cf_ib_2 & \text{ if $\phi_i(b_1) \le \varepsilon_i(b_2)$}.
\end{cases}
\end{align}


Kirillov-Reshetikhin modules  $W^{r,s}$ are finite-dimensional $U'_q(\g)$-modules
parameterized by $(r,s) \in I_0 \times \ZZ_{\ge 1}$.
For nonexceptional $\g$, the  $W^{r,s}$ have crystal pseudobases \cite{KKMMNNDuke, OkadotypeD, OkadoSchillingnonexcept}, and these
yield $U'_q(\g)$-crystals $B^{r,s}$ known as KR crystals.
We are interested in the subclass of
perfect KR crystals (see \cite{KKMMNN}); we will not work with the definition directly, but only need the following from  \cite{FOSperfect}:
a KR crystal  $B^{r,s}$ is perfect if and only if $c_r = \max\{1, a_r/a_r^\vee\}$ divides~$s$.

\subsection{Dynkin diagram automorphisms and crystals}
\label{ss twisiting}

For $\tau \in \Sigma$ and
$U_q(\g)$-crystals (resp. $U'_q(\g)$-crystals)  $B, B'$,
a bijection of sets $z \colon B \to B' $ is a  \emph{$\tau$-twist} if
\begin{align*}
\tau(\wt(b)) &= \wt(z(b)), \ \, \text{ and } \\
  z(\ce_{i}b) &=\ce_{\tau(i)}z(b), \ \   z(\cf_{i}b)= \cf_{\tau(i)}z(b) \ \text{ for all  $i \in I$, \ where  $z(0) := 0.$}
\end{align*}
Since $\tau(P) = P$ and $\tau(\delta) = \delta$, $\tau$ yields automorphisms of  $P$ and $P_{\cl}$, and thus $\tau(\wt(b))$ belongs to $P$ (resp.  $P_{\cl}$).

\begin{proposition}[{\cite[Lemma 6.5]{SchillingTingley}}, {\cite[Proposition 5.5]{NaoiVariousLevels}}]
\label{p KR twist}
For any KR crystal  $B$ and $\tau \in \Sigma$, there exists a unique  $\tau$-twist
 of $U_q'(\g)$-crystals
$\F_\tau^B \colon  B \to B$.
\end{proposition}

There is also a unique  $\tau$-twist
$\F^\Lambda_\tau \colon B(\Lambda) \to B(\tau(\Lambda))$
for any $\Lambda \in P^+$, where $B(\Lambda)$ is the
$U_q(\g)$-crystal of the irreducible  $U_q(\g)$-module of highest weight  $\Lambda$ \cite{Kas1,KashiwaraSurvey}.

It is easily verified that if  $z_1 \colon B_1 \to B_1'$ and  $z_2 \colon B_2 \to B_2'$ are  $\tau$-twists,
then  so is $z_1 \tsr z_2 \colon B_1 \tsr B_2 \to B_1' \tsr B_2'$.
Thus  the tensor product of maps
$\F^{\Lambda^1}_\tau \tsr \F^{\Lambda^2}_\tau$
is the natural choice of  $\tau$-twist
from any tensor product  $B(\Lambda^1) \tsr B(\Lambda^2)$ of  highest weight $U_q(\g)$-crystals, $\Lambda^1,\Lambda^2 \in P^+$.
Using in addition Proposition \ref{p KR twist},
a similar  \emph{$\tau$-twist exists from any tensor product of KR crystals and highest weight crystals to another such product, and we denote it  $\F_\tau$} (these are the only crystals we will consider in this paper);
for example, for a KR crystal  $B$, we denote by $\F_\tau$ the map  $\F_\tau^{\Lambda_0} \tsr \F^B_\tau \colon
B(\Lambda_0) \tsr B \to B(\Lambda_{\tau(0)}) \tsr B$, where  $B(\Lambda_0)$ and  $B(\Lambda_{\tau(0)})$ are regarded as  $U'_q(\g)$-crystals by restriction.

\subsection{Generalized Demazure crystals}
For a crystal $B$, $S \subset B$,
and $i \in I$, define
\[\F_i \spa S := \{\cf_i^k b \mid b\in S, k \ge 0 \} \setminus \{0\} \subset B.\]
For a reduced expression $w = s_{i_1} \cdots s_{i_m} \in W$,
we write  $\F_w S$ for  $\F_{i_1} \cdots \F_{i_m} S$ when this is well defined, i.e., does not depend on the choice of reduced expression.
A \emph{Demazure crystal} is a subset of some highest weight $U_q(\g)$-crystal  $B(\Lambda)$
of the form $B_w(\Lambda) := \F_{w} \{u_{\Lambda}\}$ for some  $w \in W$, where $u_\Lambda$ is the
highest weight element of  $B(\Lambda)$;
it is well defined by \cite{KashiwaraDemazure}.

A  \emph{generalized Demazure crystal} is a subset of
a tensor product of highest weight crystals of the form
$\F_{w_1} \F_{\tau_1}\big( u_{\Lambda^1} \tsr \F_{w_{2}}\F_{\tau_2} \big( u_{\Lambda^{2}}  \tsr \cdots \F_{w_p} \F_{\tau_p}(\{u_{\Lambda^p}\}) \cdots \big) \big)$
for some $\Lambda^1, \dots, \Lambda^p \in P^+$, $w_1, \dots, w_p \in W$, and  $\tau_1, \dots, \tau_p \in \Sigma$.
The combinatorial excellent filtration theorem \cite{LakLitMag}, \cite{Josephdemazurecrystal} 
states that $u_{\Lambda^1} \tsr \F_{w_{2}}\{u_{\Lambda^{2}}\} \subset B(\Lambda^1) \tsr B(\Lambda^2)$
is a disjoint union of Demazure crystals.  It follows that (see \cite[Lemma 4.3]{NaoiVariousLevels})

\begin{theorem}
\label{t monomial times Demazure}
Any generalized Demazure crystal is a disjoint union of Demazure crystals
(and the above expression is well defined).
\end{theorem}

\section{Matching generalized Demazure and DARK crystals}

\begin{lemma}[\cite{Kas1}]
\label{l tensor injection}
For any  $\Lambda, \Lambda' \in P^+$, there is a strict embedding of $U_q(\g)$-crystals \break
$B(\Lambda + \Lambda') \hookrightarrow B(\Lambda) \tsr B(\Lambda')$ determined by $u_{\Lambda + \Lambda'} \mapsto u_\Lambda \tsr u_{\Lambda'}$;
it maps  $B(\Lambda + \Lambda')$ isomorphically onto a connected component of $B(\Lambda) \tsr B(\Lambda')$.
\end{lemma}
\begin{proof}
It is well known \cite{Kas1} that $B(\Lambda) \tsr B(\Lambda')$ is isomorphic to a disjoint union of highest weight crystals.
Since $\ce_i(u_\Lambda \tsr u_{\Lambda'}) = 0$  for all $i \in I$, it is the highest weight element of a connected component isomorphic to $B(\Lambda + \Lambda')$.
\end{proof}

The following result gives a beautiful connection between Demazure and KR crystals;
part (i) is due to \cite{KKMMNN} and (ii) to \cite[Theorem 6.1]{SchillingTingley}.
Let $w_0$ be the longest element of $W_0$.

\begin{theorem}
\label{t Schilling Tingley}
Let $B = B^{r,c_r s}$ be a perfect KR crystal.
There is a unique element $\sfb^{r,s} \in B$ satisfying
$\varepsilon_0(\sfb^{r,s}) = s$ and $\varepsilon_i(\sfb^{r,s}) = 0$ for $i \in I_0$.
Put $\mu = c_rw_0(\varpi_r)$ and write $t_\mu = y \tau$ with $y \in W$, $\tau \in \Sigma$.
\begin{list}{\emph{(\roman{ctr})}} {\usecounter{ctr} \setlength{\itemsep}{1pt} \setlength{\topsep}{2pt}}
\item There is a $U_q'(\g)$-crystal isomorphism
    \[\theta \colon  B(s\Lambda_0) \tsr B \xrightarrow{\cong} B(s\Lambda_{\tau(0)})\]
    which maps  $u_{s\Lambda_0} \tsr \sfb^{r,s} \mapsto u_{s\Lambda_{\tau(0)}}$.
    Here, $B(s\Lambda_0)$ and $B(s\Lambda_{\tau(0)})$ are regarded as $U'_q(\g)$-crystals by restriction---see \S\ref{ss crystals}.
\item  $\theta$ maps the subset $u_{s\Lambda_0} \tsr B$ onto the Demazure crystal $B_y(s\Lambda_{\tau(0)}) \subset B(s\Lambda_{\tau(0)})$.
\end{list}
\end{theorem}

\begin{remark}
It is convenient to allow  $s=0$ in
Theorem \ref{t Schilling Tingley},
which holds (trivially) with $B^{r,0} = \{\sfb^{r,0}\}$ defined to be the trivial  $U'_q(\g)$-crystal, i.e., $\wt(\sfb^{r,0})=0$ and $\ce_i(\sfb^{r,0})= \cf_i(\sfb^{r,0}) = 0$ for all  $i\in I$.
Note that  $B(0\Lambda_i)= B(0)= \{u_0\}$ is the trivial $U_q(\g)$-crystal.
\end{remark}

\begin{lemma}
\label{l lemma tsr help}
Let  $A$ and  $Z$ be  $U'_q(\g)$-crystals.
Let  $u \in A$, $z \in Z$, and $j_1,\dots, j_m \in I$,
and suppose that $\F_{j_1}\cdots\F_{j_m}(u \tsr z ) \subset u \tsr Z$ in  $A \tsr Z$.
Then for any $G = \cf_{j_{1}}^{d_{1}} \cdots \cf_{j_{m}}^{d_{m}}$,  $d_i \in \ZZ_{\ge 0}$,
$G (u \tsr z) = u \tsr G(z)$.
\end{lemma}
\begin{proof}
The containment tells us that 
each application of $\cf_i$ in computing  $G(u \tsr z)$ can only be applied to $u$ if  $\cf_i (u) =0$, but this would mean $\phi_i(u) = 0$ and  $\cf_i$ is applied on the left tensor factor, which is not allowed by the tensor product rule \eqref{e crystal tensor 2}.
\end{proof}

\begin{lemma}
\label{l lemma tsr help 2}
Maintain the notation of Theorem \ref{t Schilling Tingley} and in addition
let  $w \le y$ in Bruhat order and let $w = s_{j_1} \cdots s_{j_m}$ be a reduced expression for  $w$.
Let  $C$ be any  $U'_q(\g)$-crystal.
Then for any $G = \cf_{j_{1}}^{d_{1}} \cdots \cf_{j_{m}}^{d_{m}}$, $d_i \in \ZZ_{\ge 0}$, and $c \in C$,
there holds
$G( u_{s\Lambda_0} \tsr \sfb^{r,s} \tsr c)  = u_{s\Lambda_0} \tsr G(\sfb^{r,s} \tsr c)$ in the $U'_q(\g)$-crystal $B(s \Lambda_0) \tsr  B  \tsr  C$.
\end{lemma}
\begin{proof}
By Theorem \ref{t Schilling Tingley},
$u_{s\Lambda_0} \tsr B = \theta^{-1} (\F_y(u_{s\Lambda_{\tau(0)}})) = \F_y(\theta^{-1} (u_{s\Lambda_{\tau(0)}})) = \F_y(u_{s\Lambda_0} \tsr \sfb^{r,s})$.
Hence $\F_{j_1}\cdots\F_{j_m}(u_{s\Lambda_0} \tsr \sfb^{r,s} ) \subset \F_y(u_{s\Lambda_0} \tsr \sfb^{r,s} )  = u_{s\Lambda_0} \tsr B$ in  $B(s \Lambda_0) \tsr B$.
This implies
  $\F_{j_1}\cdots\F_{j_m} (u_{s\Lambda_0} \tsr \sfb^{r,s} \tsr c) \subset
 \F_{j_1}\cdots\F_{j_m} (u_{s\Lambda_0} \tsr \sfb^{r,s}) \tsr \F_{j_1}\cdots\F_{j_m} (c)
 \subset u_{s\Lambda_0} \tsr B \tsr C$.
The result then follows from Lemma \ref{l lemma tsr help} with  $A = B(s\Lambda_0)$, $Z = B \tsr C$.
\end{proof}

\begin{theorem}
\label{t KR to affine iso gentype}
Let $B_j = B^{\myr_j,c_{\myr_j}\lambda_j}$ for $j \in [p]$ be perfect KR crystals
with  $\bdr = (\myr_1,\dots, \myr_p) \in (I_0)^p$ and
$\lambda = (\lambda_1 \geq \lambda_2 \geq \cdots \geq \lambda_p \geq 0)$, and set $\lambda^j = \lambda_j - \lambda_{j+1}$ with $\lambda_{p+1} = 0$.
Put $\mu_j = c_{\myr_j}\omega_0(\varpi_{\myr_j})$ and write $t_{\mu_j} = y_j\tau_{j}$ with $y_j \in W$ and $\tau_j \in \Sigma$.
There is a strict embedding (see \S\ref{ss crystals}) of $U_q'(\g)$-crystals
\begin{align}
\label{et KR to affine iso gentype}
\Theta_{\bdr, \lambda}  \colon   B(\lambda_1\Lambda_0) \tsr B_1 \tsr \cdots \tsr B_p
\to
B(\lambda^1\Lambda_{\tau_{1}(0)}) \tsr \cdots \tsr B(\lambda^p \Lambda_{\tau_{1}\tau_{2} \cdots \tau_{p}(0)}).
\end{align}
\end{theorem}
\begin{proof}
Apply the isomorphism  $\theta$ of Theorem \ref{t Schilling Tingley} to the left two factors, then the strict embedding of Lemma \ref{l tensor injection}, then
apply $\F_{\tau_{1}}\theta \F_{\tau_{1}}^{-1}$ to the second and third factors, and so on:
\begin{align*}
&B(\lambda_1\Lambda_0) \tsr B_1 \tsr B_2\tsr \cdots \tsr B_p\\
\xrightarrow{\cong} \
& B(\lambda_1\Lambda_{\tau_{1}(0)}) \tsr B_2 \tsr \cdots \tsr B_p\\
\hookrightarrow \
& B(\lambda^1\Lambda_{\tau_{1}(0)}) \tsr B(\lambda_2\Lambda_{\tau_{1}(0)}) \tsr B_2 \tsr B_3 \tsr \cdots \tsr B_p \\
\xrightarrow{\cong} \
& B(\lambda^1\Lambda_{\tau_{1}(0)}) \tsr B(\lambda_2\Lambda_{\tau_{1}\tau_{2}(0)}) \tsr B_3 \tsr \cdots \tsr B_p \\
\hookrightarrow \
& B(\lambda^1\Lambda_{\tau_{1}(0)}) \tsr B(\lambda^2 \Lambda_{\tau_{1}\tau_{2}(0)}) \tsr B(\lambda_3 \Lambda_{\tau_{1}\tau_{2}(0)}) \tsr B_3 \tsr \cdots \tsr B_p \\
\cdots \\
 \to \ & B(\lambda^1\Lambda_{\tau_{1}(0)}) \tsr B(\lambda^2\Lambda_{\tau_{1}\tau_{2}(0)}) \tsr \cdots \tsr B(\lambda^p \Lambda_{\tau_{1}\tau_{2} \cdots \tau_{p}(0)}). \qedhere
\end{align*}
\end{proof}

\begin{theorem}
\label{t KR to affine iso subsets gentype}
Maintain the notation of Theorem \ref{t KR to affine iso gentype} and in addition let $w_1,\dots, w_p \in W$ with $w_i  \le y_i$ for all  $i$.
Then the subset
\begin{align}
\label{et KR to affine iso subsets gentype}
\B := \F_{w_1}
\big( \sfb^{\myr_1,\lambda_1} \tsr \F_{\tau_1} \F_{w_{2}} \big( \sfb^{\myr_2,\lambda_2} \tsr \cdots \F_{\tau_{p-1}} \F_{w_p} (\{\sfb^{\myr_p,\lambda_p}\})  \cdots \big) \big)
\subset B_1 \tsr \cdots \tsr B_p
\end{align}
is well defined---we call it a
DARK crystal---and the image of $u_{\lambda_{1} \Lambda_0} \tsr  \B$ under the strict embedding $\Theta_{\bdr, \lambda}$ is a generalized Demazure crystal:
\begin{align}
\label{et KR to affine iso subsets gentype 2}
\Theta_{\bdr, \lambda}(u_{\lambda_{1} \Lambda_0} \tsr  \B) =  \F_{w_1} \F_{\tau_1} \big( u_{\lambda^1\Lambda_{0}} \tsr \F_{w_{2}} \F_{\tau_2}  \big( u_{\lambda^2\Lambda_{0}} \tsr \cdots
 \F_{w_p} \F_{\tau_p}  (\{u_{\lambda^p\Lambda_0 }\})  \cdots \big)\big).
\end{align}
\end{theorem}
\begin{proof}
Choose reduced expressions
$w_i = s_{j_{i,1}} \cdots s_{j_{i,m_i}}$ for all $i \in [p]$.
For now, interpret $\F_{w_i}$ in \eqref{et KR to affine iso subsets gentype} as  $\F_{j_{i,1}} \cdots \F_{j_{i,m_i}}$.
Then we can specify an arbitrary element of  $u_{\lambda_{1} \Lambda_0} \tsr \B$
as in \eqref{e Theta map explicit gentype} by choosing arbitrary
$G_i = \cf_{j_{i,1}}^{d_{i,1}} \cdots \cf_{j_{i,m_i}}^{d_{i,m_i}}$ with $d_{i,1}, \dots, d_{i,m_i} \in \ZZ_{\ge 0}$ for $i \in [p]$.
Tracing through the maps making up  $\Theta_{\bdr, \lambda}$, we obtain
\begin{align}
\label{e Theta map explicit gentype}
& u_{\lambda_{1} \Lambda_0} \tsr G_{1} \Big( \sfb^{\myr_1,\lambda_1} \tsr \F_{\tau_1} G_{2} \big( \sfb^{\myr_2,\lambda_2} \tsr \cdots \F_{\tau_{p-1}} G_{p} (\sfb^{\myr_p,\lambda_p}) \nocdots \big) \Big) \\
= \ &  G_{1} \Big( u_{\lambda_{1} \Lambda_{0}} \tsr  \sfb^{\myr_1,\lambda_1} \tsr \F_{\tau_1} G_{2} \big( \sfb^{\myr_2,\lambda_2} \tsr \cdots \F_{\tau_{p-1}} G_{p} (\sfb^{\myr_p,\lambda_p}) \nocdots \big) \Big)   &&  \text{by Lemma \ref{l lemma tsr help 2}}\notag  \\
\mapsto \  &  G_{1} \Big( u_{\lambda_{1} \Lambda_{\tau_1(0)}} \tsr \F_{\tau_1} G_{2} \big( \sfb^{\myr_2,\lambda_2} \tsr \cdots \F_{\tau_{p-1}} G_{p} (\sfb^{\myr_p,\lambda_p}) \nocdots \big) \Big)   &&  \text{by Theorem \ref{t Schilling Tingley} (i)} \notag\\
= \  &  G_{1} \F_{\tau_1} \Big( u_{\lambda_{1} \Lambda_{0}} \tsr G_{2} \big( \sfb^{\myr_2,\lambda_2} \tsr \cdots \F_{\tau_{p-1}} G_{p} (\sfb^{\myr_p,\lambda_p}) \nocdots \big) \Big)
&&  \text{by \S\ref{ss twisiting}}\notag\\
\mapsto \ &  G_{1}\F_{\tau_1}   \Big( u_{\lambda^{1} \Lambda_{0}} \tsr u_{\lambda_{2} \Lambda_{0}} \tsr G_{2} \big( \sfb^{\myr_2,\lambda_2} \tsr \cdots \F_{\tau_{p-1}} G_{p} (\sfb^{\myr_p,\lambda_p})
\nocdots \big) \Big)  &&  \text{by Lemma \ref{l tensor injection}} \notag\\
= \ &  G_{1} \F_{\tau_1} \Big( u_{\lambda^{1} \Lambda_{0}} \tsr  G_2 \big( u_{\lambda_{2} \Lambda_{0}} \tsr \sfb^{\myr_2,\lambda_2} \tsr \cdots \F_{\tau_{p-1}} G_{p} (\sfb^{\myr_p,\lambda_p})
\nocdots \big)\Big)  &&  \text{by Lemma \ref{l lemma tsr help 2}} \notag\\
\mapsto \  &  G_{1} \F_{\tau_1}  \Big( u_{\lambda^{1} \Lambda_{0}} \tsr G_2 \Big( u_{\lambda_{2} \Lambda_{\tau_2(0)}} \tsr \F_{\tau_2}G_3\big(  \cdots \F_{\tau_{p-1}} G_{p} (\sfb^{\myr_p,\lambda_p}) \nocdots \big)\Big) \Big)\notag  && \text{by Theorem \ref{t Schilling Tingley} (i)}\\
\cdots \notag\\
\mapsto \  &  G_{1} \F_{\tau_1} \Big( u_{\lambda^1\Lambda_{0}} \tsr G_{2} \F_{\tau_2} \Big( u_{\lambda^2\Lambda_{0}} \tsr G_3\tau_3 \big( \cdots
G_{p}\F_{\tau_p} (u_{\lambda^p\Lambda_0}) \big) \nocdots \Big)\Big), \notag
\end{align}
which is an arbitrary element of the right side of \eqref{et KR to affine iso subsets gentype 2}.
Moreover, by Theorem \ref{t monomial times Demazure}, the right side of \eqref{et KR to affine iso subsets gentype 2} does not depend on the chosen reduced expressions for  $w_i$,
so the same goes for  $\B$ since  $\Theta_{\bdr, \lambda}$ is injective.
\end{proof}

Let $\ZZ[P]$ denote the group ring of $P$ with  $\ZZ$-basis $\{ e^\mu \}_{\mu \in P}$.
The \emph{Demazure operators} are linear operators $D_i$ on
 $\ZZ[P]$ defined for each $i \in I$ by
$D_i(f) = \frac{f-e^{-\alpha_i}\cdot s_i(f)}{1-e^{-\alpha_i}},$
where $s_i$ acts on $\ZZ[P]$ by $s_i(e^\mu) = e^{s_i(\mu)}$.
We also have an action of  $\Sigma$ on  $\ZZ[P]$ given by $\tau(e^\mu) = e^{\tau(\mu)}$.
For a reduced expression $w = s_{i_1} \cdots s_{i_m} \in W$, define the operator $D_w = D_{i_1} \cdots D_{i_m}$ on $\ZZ[P]$; it is independent of the choice of reduced expression \cite[Corollary 8.2.10]{Kumarbook}.

Naoi \cite[Theorem 7.1]{NaoiVariousLevels} showed that
$\Theta_{\bdr,\lambda}$ matches the statistic $\langle d, \wt(b) \rangle$ on $U_q(\g)$-crystals
to energy.
Combining this with Theorem \ref{t monomial times Demazure} and \cite[Corollary 4.6]{NaoiVariousLevels} we obtain

\begin{corollary}
\label{c with energy}
Maintain the notation of Theorem \ref{t KR to affine iso subsets gentype}.
The energy adjusted character of the DARK crystal $\B$ agrees with the character of
the generalized Demazure crystal in \eqref{et KR to affine iso subsets gentype 2} (call it $\B'$), and both have the following
Demazure operator formula:
\begin{align*}
e^{\lambda_1\Lambda_0+ \delta C} \sum_{b \in \B} e^{\aff( \wt(b)) - \delta \frac{D(b)}{a_0} } =
\sum_{b \in \B'} e^{\wt(b)}
= D_{w_1} \tau_1\big( e^{\lambda^1\Lambda_0} \cdot D_{w_2} \tau_2\big(e^{\lambda^2\Lambda_0} \cdots  D_{w_p} \tau_p(e^{\lambda^p\Lambda_0})\big)\big),
\end{align*}
where $D(b)$ is the energy of $b$ and $C \in \QQ$ is a constant which depends only on $\lambda$ and~$\bdr$.
\end{corollary}

\begin{remark}
When  $w_i =y_i$ for all  $i\in [p]$,  the DARK crystal  $\B$ in
\eqref{et KR to affine iso subsets gentype} is equal to $B_1 \tsr \cdots \tsr B_p$ (this follows from  $\F_{y_i} \{\sfb^{\myr_i,\lambda_i}\} = B_i$
and \cite[Lemma 5.15]{NaoiVariousLevels}). Thus  Proposition~5.16/Corollary 7.2 of \cite{NaoiVariousLevels}
are encompassed by
Theorem \ref{t KR to affine iso subsets gentype}/Corollary \ref{c with energy}.
Note that the  $a_0$ appearing in Corollary \ref{c with energy}
corrects a typo in \cite[Corollary 7.2]{NaoiVariousLevels}.
\end{remark}

\begin{remark}
\label{r extra w0}
In Theorem \ref{t KR to affine iso subsets gentype} and Corollary \ref{c with energy},
we can more generally allow $w_i$ of the form $w_i = v_i w'_i$ where  $v_i \in W_0$ and  $w_i' \le y_i$.
Indeed, in the setting of Lemma \ref{l lemma tsr help 2}, for any  $j \in I_0$,
$\cf_j G( u_{s\Lambda_0} \tsr \sfb^{r,s} \tsr c)  =
\cf_j (u_{s\Lambda_0} \tsr G( \sfb^{r,s} \tsr c) )
= u_{s\Lambda_0} \tsr \cf_j G(\sfb^{r,s} \tsr c)$
as $\cf_j(u_{s\Lambda_0})$ $= \ce_j(u_{s\Lambda_0})= 0$;
hence the lemma holds more generally with $w = vw'$ with  $v \in W_0$,  $w'\le y$.
\end{remark}

For $\g$ of type $A_n^{(1)}$ and $\bdr = \mathbf{1} = (1,\dots, 1)$,
Theorems \ref{t KR to affine iso gentype} and \ref{t KR to affine iso subsets gentype} and Remark \ref{r extra w0} give
\begin{corollary}
\label{t KR to affine iso type A}
Let $\lambda = (\lambda_1 \ge \cdots \ge \lambda_p \ge 0)$ and set $\lambda^j = \lambda_j - \lambda_{j+1}$ with $\lambda_{p+1} = 0$.
Let $\tau$ be the Dynkin diagram automorphism given by $j \mapsto j+1 \bmod n+1$.
Then there is a strict embedding of $U_q'(\hatsl_{n+1})$-crystals 
\begin{align}
\label{et KR to affine iso type A}
\Theta_{\mathbf{1}, \lambda}  \colon   B(\lambda_1\Lambda_0) \tsr B^{1,\lambda_1} \tsr \cdots \tsr B^{1, \lambda_p}
\to
B(\lambda^1\Lambda_{1}) \tsr \cdots \tsr B(\lambda^p \Lambda_p).
\end{align}
Moreover, for any $w_1, \ldots, w_p \in W_0$,
\begin{align*}
& u_{\lambda_{1} \Lambda_0} \tsr \F_{w_1  }
\big( \sfb^{1,\lambda_1} \tsr \F_{\tau} \F_{w_{2}} \big( \sfb^{1,\lambda_2} \tsr \cdots \F_{\tau} \F_{w_p} (\{\sfb^{1,\lambda_p}\})  \cdots \big) \big) \\
&\xmapsto{\Theta_{\mathbf{1}, \lambda}} \
\F_{w_1} \big( u_{\lambda^1\Lambda_{1}} \tsr \F_{\tau} \F_{w_{2}}  \big( u_{\lambda^2\Lambda_{1}} \tsr \cdots
\F_{\tau} \F_{w_p}   (\{u_{\lambda^p\Lambda_1 }\})  \cdots \big)\big).
\end{align*}
\end{corollary}
This result is used in \cite{BMP} to connect the katabolism operations of Lascoux \cite{La} and Shimozono-Weyman \cite{SW} to generalized Demazure crystals.
The combinatorial significance of Theorem \ref{t KR to affine iso subsets gentype} for more general  $\bdr$ and in other types remains to be explored.

\vspace{2mm}
\noindent
\textbf{Acknowledgments.}
We thank Katsuyuki Naoi and Jennifer Morse for helpful discussions
and Elaine~So for help typing.

\bibliographystyle{plain}
\bibliography{mycitations}

\def\cprime{$'$} \def\cprime{$'$} \def\cprime{$'$}
\begin{thebibliography}{10}

\bibitem{BMP}
Jonah {Blasiak}, Jennifer {Morse}, and Anna {Pun}.
\newblock {Demazure crystals and the Schur positivity of Catalan functions}.
\newblock {\em {\tt arXiv:2007.04952}}, July 2020.

\bibitem{ChenThesis}
Li-Chung Chen.
\newblock {\em Skew-Linked Partitions and a Representation-Theoretic Model for
  $k$-Schur Functions}.
\newblock PhD thesis, UC Berkeley, 2010.

\bibitem{FOSperfect}
Ghislain Fourier, Masato Okado, and Anne Schilling.
\newblock Perfectness of {K}irillov-{R}eshetikhin crystals for nonexceptional
  types.
\newblock In {\em Quantum affine algebras, extended affine {L}ie algebras, and
  their applications}, volume 506 of {\em Contemp. Math.}, pages 127--143.
  Amer. Math. Soc., Providence, RI, 2010.

\bibitem{FSS}
Ghislain Fourier, Anne Schilling, and Mark Shimozono.
\newblock Demazure structure inside {K}irillov-{R}eshetikhin crystals.
\newblock {\em J. Algebra}, 309(1):386--404, 2007.

\bibitem{Josephdemazurecrystal}
Anthony Joseph.
\newblock A decomposition theorem for {D}emazure crystals.
\newblock {\em J. Algebra}, 265(2):562--578, 2003.

\bibitem{KacBook}
Victor~G. Kac.
\newblock {\em Infinite-dimensional {L}ie algebras}.
\newblock Cambridge University Press, Cambridge, third edition, 1990.

\bibitem{Kangyoungwalls}
Seok-Jin Kang.
\newblock Crystal bases for quantum affine algebras and combinatorics of
  {Y}oung walls.
\newblock {\em Proc. London Math. Soc. (3)}, 86(1):29--69, 2003.

\bibitem{KKMMNN}
Seok-Jin Kang, Masaki Kashiwara, Kailash~C. Misra, Tetsuji Miwa, Toshiki
  Nakashima, and Atsushi Nakayashiki.
\newblock Affine crystals and vertex models.
\newblock In {\em Infinite analysis, {P}art {A}, {B} ({K}yoto, 1991)},
  volume~16 of {\em Adv. Ser. Math. Phys.}, pages 449--484. World Sci. Publ.,
  River Edge, NJ, 1992.

\bibitem{KKMMNNDuke}
Seok-Jin Kang, Masaki Kashiwara, Kailash~C. Misra, Tetsuji Miwa, Toshiki
  Nakashima, and Atsushi Nakayashiki.
\newblock Perfect crystals of quantum affine {L}ie algebras.
\newblock {\em Duke Math. J.}, 68(3):499--607, 1992.

\bibitem{Kas1}
Masaki Kashiwara.
\newblock On crystal bases of the {$Q$}-analogue of universal enveloping
  algebras.
\newblock {\em Duke Math. J.}, 63(2):465--516, 1991.

\bibitem{KashiwaraDemazure}
Masaki Kashiwara.
\newblock The crystal base and {L}ittelmann's refined {D}emazure character
  formula.
\newblock {\em Duke Math. J.}, 71(3):839--858, 1993.

\bibitem{KashiwaraSurvey}
Masaki Kashiwara.
\newblock On crystal bases.
\newblock In {\em Representations of groups ({B}anff, {AB}, 1994)}, volume~16
  of {\em CMS Conf. Proc.}, pages 155--197. Amer. Math. Soc., Providence, RI,
  1995.

\bibitem{Kumarbook}
Shrawan Kumar.
\newblock {\em Kac-{M}oody groups, their flag varieties and representation
  theory}, volume 204 of {\em Progress in Mathematics}.
\newblock Birkh\"{a}user Boston, Inc., Boston, MA, 2002.

\bibitem{LakLitMag}
Venkatramani Lakshmibai, Peter Littelmann, and Peter Magyar.
\newblock Standard monomial theory for {B}ott-{S}amelson varieties.
\newblock {\em Compositio Math.}, 130(3):293--318, 2002.

\bibitem{La}
Alain Lascoux.
\newblock Cyclic permutations on words, tableaux and harmonic polynomials.
\newblock In {\em Proceedings of the {H}yderabad {C}onference on {A}lgebraic
  {G}roups ({H}yderabad, 1989)}, pages 323--347, Madras, 1991. Manoj Prakashan.

\bibitem{LenartPostnikov}
Cristian Lenart and Alexander Postnikov.
\newblock A combinatorial model for crystals of {K}ac-{M}oody algebras.
\newblock {\em Trans. Amer. Math. Soc.}, 360(8):4349--4381, 2008.

\bibitem{LittelmannCrystal}
Peter Littelmann.
\newblock Crystal graphs and {Y}oung tableaux.
\newblock {\em J. Algebra}, 175(1):65--87, 1995.

\bibitem{Littelmannpaths}
Peter Littelmann.
\newblock Paths and root operators in representation theory.
\newblock {\em Ann. of Math. (2)}, 142(3):499--525, 1995.

\bibitem{NaoiXMConjecture}
Katsuyuki Naoi.
\newblock Fusion products of {K}irillov-{R}eshetikhin modules and the {$X=M$}
  conjecture.
\newblock {\em Adv. Math.}, 231(3-4):1546--1571, 2012.

\bibitem{NaoiVariousLevels}
Katsuyuki Naoi.
\newblock Demazure crystals and tensor products of perfect
  {K}irillov-{R}eshetikhin crystals with various levels.
\newblock {\em J. Algebra}, 374:1--26, 2013.

\bibitem{OkadotypeD}
Masato Okado.
\newblock Existence of crystal bases for {K}irillov-{R}eshetikhin modules of
  type {$D$}.
\newblock {\em Publ. Res. Inst. Math. Sci.}, 43(4):977--1004, 2007.

\bibitem{OkadoSchillingnonexcept}
Masato Okado and Anne Schilling.
\newblock Existence of {K}irillov-{R}eshetikhin crystals for nonexceptional
  types.
\newblock {\em Represent. Theory}, 12:186--207, 2008.

\bibitem{SchillingTingley}
Anne Schilling and Peter Tingley.
\newblock Demazure crystals, {K}irillov-{R}eshetikhin crystals, and the energy
  function.
\newblock {\em Electron. J. Combin.}, 19(2):Paper 4, 42, 2012.
\newblock [Second author's name now ``Tingley'' on article].

\bibitem{Shimozonoaffine}
Mark Shimozono.
\newblock Affine type {A} crystal structure on tensor products of rectangles,
  {D}emazure characters, and nilpotent varieties.
\newblock {\em J. Algebraic Combin.}, 15(2):151--187, 2002.

\bibitem{SW}
Mark Shimozono and Jerzy Weyman.
\newblock Graded characters of modules supported in the closure of a nilpotent
  conjugacy class.
\newblock {\em European J. Combin.}, 21(2):257--288, 2000.

\end{thebibliography}
\def\cprime{$'$} \def\cprime{$'$} \def\cprime{$'$}

\end{document}